\documentclass[preprint,11pt,sort&compress,numbers]{elsarticle}

\usepackage{amsmath,amsthm,amsfonts,amssymb,latexsym,mathrsfs,url}

\newtheorem{thm}{Theorem}[section]
\newtheorem{lem}[thm]{Lemma}
\newtheorem{coro}[thm]{Corollary}

\theoremstyle{definition}

\newtheorem{exm}[thm]{Example}
\newtheorem{rem}[thm]{Remark}

\newcommand{\lrf}[1]{\lfloor #1\rfloor}

\numberwithin{equation}{section}

\journal{JMAA}

\begin{document}

\begin{frontmatter}

\title{Asymptotic normality of Laplacian coefficients of graphs}

\author[a]{Yi Wang\corref{cor1}}
\ead{wangyi@dlut.edu.cn}
\cortext[cor1]{Corresponding authors.}
\author[b]{Hai-Xia Zhang}
\ead{zhanghaixiass@hotmail.com}
\author[c]{Bao-Xuan Zhu\corref{cor1}}
\ead{bxzhu@jsnu.edu.cn}

\address[a]{School of Mathematical Sciences, Dalian University of Technology, Dalian 116024, P.R. China}
\address[b]{Department of Mathematics, Taiyuan University of Science and Technology, Taiyuan 030024, P.R. China}
\address[c]{School of Mathematics and Statistics, Jiangsu Normal University, Xuzhou 221116, P.R. China}

\begin{abstract}
Let $G$ be a simple graph with $n$ vertices and let
$$C(G;x)=\sum_{k=0}^n(-1)^{n-k}c(G,k)x^k$$
denote the Laplacian characteristic polynomial of $G$.
Then if the size $|E(G)|$ is large compared to the maximum degree $\Delta(G)$,
Laplacian coefficients $c(G,k)$ are approximately normally distributed (by central and local limit theorems).
We show that Laplacian coefficients of the paths, the cycles, the stars, the wheels
and regular graphs of degree $d$ are approximately normally distributed respectively.
We also point out that Laplacian coefficients of the complete graphs and the complete bipartite graphs are approximately Poisson distributed respectively.
\end{abstract}


\end{frontmatter}

\section{Introduction}

Let $a(n,k)$ be a double-indexed sequence of nonnegative numbers and let
\begin{equation}\label{pnk}
p(n,k)=\frac{a(n,k)}{\sum_{j=0}^na(n,j)}
\end{equation}
denote the normalized probabilities.
Following Bender~\cite{Ben73},
we say that the sequence $a(n,k)$ is {\it asymptotically normal by a central limit theorem},
if
\begin{equation}\label{clt}
\lim_{n\rightarrow\infty}\sup_{x\in\mathbb{R}}\left|\sum_{k\le\mu_n+x\sigma_n}p(n,k)-\frac{1}{\sqrt{2\pi}}\int_{-\infty}^xe^{-t^2/2}dt\right|=0,
\end{equation}
where $\mu_n$ and $\sigma^2_n$ are the mean and variance of \eqref{pnk}, respectively.
We say that $a(n,k)$ is {\it asymptotically normal by a local limit theorem} on $\mathbb{R}$ if
\begin{equation}\label{llt}
\lim_{n\rightarrow\infty}\sup_{x\in\mathbb{R}}\left|\sigma_np(n,\lrf{\mu_n+x\sigma_n})-\frac{1}{\sqrt{2\pi}}e^{-x^2/2}\right|=0.
\end{equation}
In this case,
$$a(n,k)\sim \frac{e^{-x^2/2}\sum_{j=0}^na(n,j)}{\sigma_n\sqrt{2\pi}} \textrm{ as } n\rightarrow \infty,$$
where $k=\mu_n+x\sigma_n$ and $x=O(1)$.
Clearly, the validity of \eqref{llt} implies that of \eqref{clt}.

Let $m(G,k)$ be the number of $k$-matchings in a graph $G$.
It is well known that the matchings generating function $M(G;x)=\sum_{k\ge 0}m(G,k)x^k$ has only real zeros.
Using this fact, Godsil \cite{God81} showed that
if $|V(G)|$ is large compared to the maximum degree $\Delta(G)$ of a vertex in $G$
or $G$ is large complete graph,
then $m(G,k)$ are approximately normally distributed
(see \cite{Kah00,LPRS16,Ruc84} for some further work).

Laplacian characteristic polynomials are also
a class of graph polynomials with only real zeros.
Given a simple graph $G$ with $n$ vertices,
its Laplacian matrix $L=L(G)$ is an $n\times n$ symmetric matrix,
defined as $L(G)=D(G)-A(G)$,
where $D(G)$ is the degree matrix and $A(G)$ is the adjacency matrix of the graph $G$.
The Laplacian characteristic polynomial of $G$ is denoted by
$$C(G;x)=\det(xI-L(G))=\sum_{k=0}^n(-1)^{n-k}c(G,k)x^k.$$
Laplacian coefficients of graphs are related to various combinatorial properties of graphs.
For example, Kelmans and Chelnokov \cite{KC74} gave an interpretation of Laplacian coefficients in terms of spanning subforest:
$$c(G,k)=\sum_{E(F)=n-k}p(F),$$
where the sum is taken over all spanning forest $F$ of $G$,
and $p(F)$ is the product of the numbers of vertices in the components of $F$.
Clearly,
$$c(G,n)=1, c(G,n-1)=2|E(G)|, c(G,0)=0, c(G,1)=n\tau(G),$$
where $\tau(G)$ is the number of the spaning trees.
In particular,
if $G$ is a tree, then the Laplacian coefficient $c(G,2)$ equals to the Wiener index $W(G)$ of $G$,
which is the sum of all distance between unordered pairs of vertices of $G$
and is considered as one of the most used indices with high correlation with many physical and chemical properties of molecular compounds
\cite{DEG01}.

Let $r$ be the number of connected components of $G$.
Then $c(G,k)=0$ if and only if $k<r$.
The object of this note is to investigate asymptotic normality
of Laplacian coefficients $c(G,k)$.
Our main result is that
if $|E(G)|$ is large compared to $\Delta(G)$,
then $c(G,k)$ are approximately normally distributed.
We show that Laplacian coefficients $c(G_n,k)$ are asymptotically normal
if $G_n$ are the paths $P_n$, the cycles $C_n$, the stars $K_{1,n}$, the binary trees $B_n$, the wheels $W_n$, and the hypercubes $Q_n$.
We also pointed out that for the complete graphs $K_n$,
Laplacian coefficients $c(K_n,k)$ are not asymptotically normal.

Throughout this paper all graphs considered are finite and simple.
Graph theoretical terms used but not defined can be found in \cite{CRS10}.

\section{Main results}

A standard approach to demonstrating asymptotic normality is the following criterion,
which was used by Harper~\cite{Har67} to show asymptotic normality of the Stirling numbers of the second kind
(see \cite[Theorem 2]{Ben73} and \cite[Example 3.4.2]{Can15} for historical remarks).

\begin{lem}\label{rzv}
Suppose that $A_n(x)=\sum_{k=0}^na(n,k)x^k$ have only real zeros and $A_n(x)=\prod_{i=1}^n(x+r_i)$,
where all $a(n,k)$ and $r_i$ are nonnegative.
Let $$\mu_n=\sum_{i=1}^n\frac{1}{1+r_i}$$
and $$\sigma^2_n=\sum_{i=1}^n\frac{r_i}{(1+r_i)^2}.$$
Then if $\sigma_n\rightarrow+\infty$,
the numbers $a(n,k)$ are asymptotically normal (by central and local limit theorems)
with the mean $\mu_n$ and variance $\sigma_n^2$.
\end{lem}

Let $G$ be a simple graph with $n$ vertices.
It is well known that the Laplacian matrix $L(G)$ is a positive semi-definite symmetric matrix.
Denote by $\lambda_i=\lambda_i(G)$ the $i$th largest eigenvalue of $L(G)$. Then
$$\lambda_1\ge\lambda_2\ge\cdots\ge\lambda_{n-1}\ge\lambda_n=0.$$
Thus
$$C(G;x)=\prod_{i=1}^n(x-\lambda_i),$$
and so $$\sum_{k=0}^nc(G,k)x^k=\prod_{i=1}^n(x+\lambda_i).$$
Therefore the associated mean and variance are
$$\mu_n=\sum_{i=1}^n\frac{1}{1+\lambda_i}$$
and
$$\sigma_n^2=\sum_{i=1}^n\frac{\lambda_i}{(1+\lambda_i)^2}.$$

\begin{exm}
Let $K_{1,n-1}$ be the star with $n$ vertices.
Then a simple evaluation yields $C(K_{1,n-1};x)=x(x-n)(x-1)^{n-2}$.
Thus the associated mean and variance are
$$\mu_n=\frac{1}{1+0}+\frac{1}{1+n}+(n-2)\frac{1}{1+1}=\frac{n^2+n+2}{2(n+1)}$$
and
$$\sigma_n^2=\frac{n}{(1+n)^2}+(n-2)\frac{1}{2^2}=\frac{(n-1)(n^2+n+2)}{4(n+1)^2}.$$
It follows that $c(K_{1,n-1},k)$ are asymptotically normal from Lemma \ref{rzv}.
\end{exm}

\begin{thm}\label{thm-dv}
Let $\{G_n\}_{n=1}^{\infty}$ be a sequence of simple graphs,
such that $|E(G_n)|$ increases with $n$,
and $\Delta^2(G_n)=o(|E(G_n)|)$.
Then the Laplacian coefficients $c(G_n,k)$ are asymptotically normal.
\end{thm}
\begin{proof}
For any graph $G$,
we have
\begin{equation}\label{kel67}
\lambda_i(G)\le 2\Delta(G)
\end{equation}
by the famous Gershgorin circle theorem.
On the other hand, it is clear that
$$\sum_{i=1}^{|V(G)|}\lambda_i(G)=\mathrm{trace}(L(G))=\sum_{u\in V(G)}\deg(u)=2|E(G)|.$$
Hence
$$\sum_{i=1}^{|V(G_n)|}\frac{\lambda_i(G_n)}{(1+\lambda_i(G_n))^2}
\ge \frac{1}{(1+2\Delta(G_n))^2}\sum_{i=1}^{|V(G_n)|}\lambda_i(G_n)
=\frac{2|E(G_n)|}{(1+2\Delta(G_n))^2},$$
and so
$$\sum_{i=1}^{|V(G_n)|}\frac{\lambda_i(G_n)}{(1+\lambda_i(G_n))^2}\rightarrow \infty$$
by the hypothesis $\Delta^2(G_n)=o(|E(G_n)|)$.
It follows that $c(G_n,k)$ are asymptotically normal from Lemma \ref{rzv}.
This completes the proof of the theorem.
\end{proof}

\begin{rem}\label{kc74}
For any graph $G$, the largest Laplacian eigenvalue satisfies
$$\lambda_1(G)\le\max\{\deg(u)+\deg(v): uv\in E(G)\}$$ (see \cite{AM85,KC74} for instance).
The inequality \eqref{kel67} is an immediate consequence of this inequality.
\end{rem}

\begin{exm}
Let $P_n$ be the path with $n$ vertices.
It is well known that
$$C(P_n;x)=\prod_{j=0}^{n-1}\left(x-4\sin^2\frac{j\pi}{2n}\right)$$
(see, e.g., \cite[\S 1.4.4]{BH12}).
Thus the associated mean $\mu_n$ and variance $\sigma^2_n$ satisfy
$$\frac{\mu_n}{n}=\frac{1}{n}\sum_{j=0}^{n-1}\frac{1}{1+4\sin^2\frac{j\pi}{2n}}
\rightarrow \frac{1}{\pi}\int_{0}^{\frac{\pi}{2}}\frac{1}{1+4\sin^2\theta}d\theta=\frac{1}{2\sqrt{5}}$$
and
$$\frac{\sigma^2_n}{n}=\frac{1}{n}\sum_{j=0}^{n-1}\frac{4\sin^2\frac{j\pi}{2n}}{\left(1+4\sin^2\frac{j\pi}{2n}\right)^2}
\rightarrow \frac{1}{\pi}\int_{0}^{\frac{\pi}{2}}\frac{4\sin^2\theta}{\left(1+4\sin^2\theta\right)^2}d\theta=\frac{1}{5\sqrt{5}}.$$
Clearly, $\Delta(P_n)\le 2$.
It follows from Theorem \ref{thm-dv} that the Laplacian coefficients $c(P_n,k)$ are asymptotically normal
with mean $\mu_n\sim\frac{n}{2\sqrt{5}}$ and variance $\sigma_n^2\sim\frac{n}{5\sqrt{5}}$.
\end{exm}

\begin{exm}
More general, let $(G_n)_{n\ge 0}$ be a sequence of (rooted) binary trees,
such that $|V(G_n)|$ increases with $n$.
Then the Laplacian coefficients $c(G_n,k)$ are asymptotically normal since $\Delta(G_n)\le 3$.
\end{exm}

The following corollary is an immediate consequence of Theorem \ref{thm-dv}.

\begin{coro}\label{thm-d}
Let $\{G_n\}_{n=1}^{\infty}$ be a sequence of graphs,
each regular of degree $d$,
such that $|V(G_n)|$ increases with $n$.
Then the Laplacian coefficients $c(G_n,k)$ are asymptotically normal.
\end{coro}

\begin{exm}
Let $G_n=nK_2$ be the disjoint union of $n$ copies of $K_2$.
Then $G_n$ is $1$-regular
and $$C(G_n;x)=x^n(x-2)^n.$$
It follows that $c(G_n,k)$ are asymptotically normal from Corollary~\ref{thm-d}.
This result is a special case of the classic de Moivre-Laplace central limit theorem.
\end{exm}

\begin{exm}
Let $C_n$ be the cycle with $n$ vertices.
Then $C_n$ is $2$-regular
and
\begin{equation}\label{clp}
C(C_n;x)=\prod_{j=0}^{n-1}\left(x-4\sin^2\frac{j\pi}{n}\right)
\end{equation}
(see, e.g., \cite[\S 1.4.3]{BH12}).
The associated mean and variance are
$$\mu_n=\sum_{j=0}^{n-1}\frac{1}{1+4\sin^2\frac{j\pi}{n}}
\rightarrow \frac{n}{\pi}\int_{0}^{\pi}\frac{1}{1+4\sin^2\theta}d\theta=\frac{n}{\sqrt{5}}$$
and
$$\sigma^2_n=\sum_{j=0}^{n-1}\frac{4\sin^2\frac{j\pi}{n}}{\left(1+4\sin^2\frac{j\pi}{n}\right)^2}
\rightarrow \frac{n}{\pi}\int_{0}^{\pi}\frac{4\sin^2\theta}{\left(1+4\sin^2\theta\right)^2}d\theta=\frac{2n}{5\sqrt{5}}.$$
Thus the Laplacian coefficients $c(C_n,k)$ are asymptotically normal with mean $\mu_n$ and variance $\sigma_n^2$ by Corollary~\ref{thm-d}.
\end{exm}

\begin{exm}
A cubic graph is a $3$-regular graph.
Many important graphs are cubic,
including interesting and mysterious snarks.
A snark is a connected, bridgeless cubic graph with chromatic index equal to $4$.
The first known snark was the Petersen graph.
Tutte conjectured that
every snark has a subgraph that can be formed from the Petersen graph by subdividing some of its edges
(this conjecture is a strengthened form of the Four-Color Theorem).
One often encounters snarks
in the study of various important and difficult problems in graph theory.
For example, the Four-Color Theorem is equivalent to the statement that no snark is planar.
Let $(G_n)_{n\ge 0}$ be a sequence of cubic graphs such that $|V(G_n)|$ increases with $n$.
Then $c(G_n,k)$ are asymptotically normal by Corollary~\ref{thm-d}.
\end{exm}

The {\it join} $G_1\nabla G_2$ of (disjoint) graphs $G_1$ and $G_2$ is the graph
obtained from $G_1\cup G_2$ by joining each vertex of $G_1$ to each vertex of $G_2$.
The graph $\widehat{G}:=G\nabla K_1$ is called {\it the cone} over the graph $G$.
It is well known that
\begin{equation*}\label{crs719}
C(G_1\nabla G_2;x)=\frac{x-n_1-n_2}{(x-n_1)(x-n_2)}C(G_1;x-n_2)C(G_2;x-n_1),
\end{equation*}
where $n_i=|V(G_i)|$ (see, e.g., \cite[Theorem 7.1.9]{CRS10}).
In particular, if $|V(G)|=n$ and $C(G;x)=x\prod_{i=1}^{n-1}(x-\lambda_i)$, then
\begin{equation}\label{cone}
C(\widehat{G};x)=x(x-n-1)\prod_{i=1}^{n-1}(x-1-\lambda_i).
\end{equation}

\begin{coro}\label{thm-dc}
Let $\{G_n\}_{n=1}^{\infty}$ be a sequence of graphs,
each regular of degree $d$,
such that $|V(G_n)|$ increases with $n$.
Then the Laplacian coefficients $c(\widehat{G_n},k)$ are asymptotically normal.
\end{coro}
\begin{proof}
Let the Laplacian spectrum of $G_n$ be $\lambda_1\,\ldots,\lambda_{N_n-1},\lambda_{N_n}=0$,
where $N_n=|V(G_n)|$.
Then by \eqref{cone},
the Laplacian spectrum of $\widehat{G_n}$ is $0,1+N_n,1+\lambda_1,\ldots,1+\lambda_{N_n-1}$.
Thus
$$\sigma_{N_n}^2(\widehat{G_n})
=\frac{0}{(0+1)^2}+\frac{1+N_n}{(2+N_n)^2}+\sum_{i=1}^{N_n-1}\frac{1+\lambda_i}{(2+\lambda_i)^2}.$$
Note that the function $\frac{1+\lambda}{(2+\lambda)^2}$ is monotonically decreasing for $\lambda\ge 0$
and that $\lambda_i\le 2d$ by Remark \ref{kc74}.
Hence
$$\sum_{i=1}^{N_n-1}\frac{1+\lambda_i}{(2+\lambda_i)^2}\ge (N_n-1)\frac{1+2d}{(2+2d)^2}.$$
Thus $\sigma_{N_n}^2(\widehat{G_n})\rightarrow\infty$ as $n\rightarrow\infty$,
and the Laplacian coefficients $c(\widehat{G_n},k)$ are therefore asymptotically normal by Lemma \ref{rzv}.
\end{proof}

\begin{exm}
Let $W_n$ be the wheels with $n+1$ vertices.
Then $W_n$ can be viewed as the cone over the cycle $C_n$.
Thus the Laplacian coefficients $c(W_n,k)$ are asymptotically normal by Corollary \ref{thm-dc}.
\end{exm}

\begin{exm}
The hypercube $Q_n$ is the Cartesian product of $n$ factors $K_2$.
Clearly, $Q_n$ is $n$-regular, with $2^n$ vertices and $n2^{n-1}$ edges.
It follows that the Laplacian coefficients $c(Q_n,k)$ are asymptotically normal from Theorem \ref{thm-dv} since $n=o(2^{n-1})$.

On the other hands,
note that the spectrum of $Q_n$ consists of the eigenvalues $n-2k$ with multiplicity $\binom{n}{k}$
(see, e.g., \cite[\S 1.4.6]{BH12}).
Hence the Laplacian spectrum of $Q_n$ consists of the eigenvalues $2k$ with multiplicity $\binom{n}{k}$.
Thus the associated variance is
$$\sum_{k=1}^n\frac{2k\binom{n}{k}}{(1+2k)^2}
\ge\sum_{k=1}^n\frac{2n\binom{n}{k}}{(1+2n)^2}
=\frac{2n(2^n-1)}{(1+2n)^2}\rightarrow\infty,$$
and the asymptotic normality of $c(Q_n,k)$ follows directly from Lemma~\ref{rzv}.
\end{exm}

\begin{rem}
The reader may wonder if there does exist a non-trivial sequence $G_n$ of graphs
for which the numbers $c(G_n,k)$ are not normally distributed.
Actually, if we take $G_n$ to be the complete graph $K_n$ with $n$ vertices,
then $C(K_n;x)=x(x-n)^{n-1}$,
and so $c(K_n,k)=n^{n-k}\binom{n-1}{k-1}$.
Thus we have
$$\frac{c(K_n,k)}{\sum_{j=1}^{n}c(K_n,j)}=\frac{n^{n-k}\binom{n-1}{k-1}}{(n+1)^{n-1}}\sim \frac{e^{-1}}{(k-1)!}$$
by Stirling's approximation for factorials.
Therefore the Laplacian coefficients $(c(K_n,k))_{k\ge 1}$ of the complete graphs $K_n$ are asymptotically Poisson distributed with the mean and variance $1$.
Similarly, we may show that the Laplacian coefficients of the complete bipartite graphs $K_{n,n}$
are asymptotically Poisson distributed with the mean and variance $2$.
\end{rem}

\section{Further work}

Laplacian coefficients are closely related to matching numbers, especially for trees (see \cite{ZG08} for instance). 
Let $S(G)$ denote the subdivision of a graph $G$, obtained by inserting a new vertex on each edge of $G$.
Zhou and Gutman \cite{ZG08} showed that
$c(T,k)=m(S(T),k)$
for every acyclic graph $T$ with $n$ vertices and $1\le k\le n$.
Using this correspondence,
they also showed that for any tree $T_n$ with $n$ vertices,
$$c(K_{1,n-1},k)\le c(T_n,k)\le c(P_n,k),\quad k=1,2,\ldots,n.$$
See Mohar \cite{Moh07} for a strengthening of this result.
We have seen that Laplacian coefficients for the stars, the paths and the binary trees are asymptotically normal respectively.
It is possible that Laplacian coefficients for the general trees are also asymptotically normal.

Another interesting topic is asymptotic normality of the signless Laplacian coefficients.
Let $Q(G)=D(G)+A(G)$ be the signless Laplacian matrix of a graph $G$
and $Q(G;x)=\det(xI-Q(G))$ the signless Laplacian characteristic polynomial.
Then $Q(G)$ is also a positive semi-definite matrix.
Let $Q(G;x)=\sum_{k=0}^n(-1)^{n-k}q(n,k)x^{k}$.
Then the signless Laplacian coefficients $q(G,k)$ may be interpreted in terms of TU-subgraphs of $G$
(see \cite{CRS07} for details).
The method of proof used in Theorem \ref{thm-dv} can be carried over verbatim to the signless Laplacian coefficients.
We leave it and related problems to the interested reader.

\section*{Acknowledgement}

This work was supported in part by the National Natural Science Foundation of China (Nos. 11371078, 11571150).
The authors thank the anonymous referee for his/her careful reading and helpful comments.

\end{document}